\documentclass{amsart}
\usepackage{amsfonts,amssymb,amsmath,amsthm}
\usepackage{url}
\usepackage{enumerate}

\newtheorem{thrm}{Theorem}[section]
\newtheorem{lem}[thrm]{Lemma}

\theoremstyle{definition}

\numberwithin{equation}{section}

\author{Wen-ming Lu and Lin Zhang}
\address{
Wen-ming Lu\newline\indent School of Science\\Hangzhou Dianzi
University\\Hangzhou\\People's Republic of
China}\email{lu\_wenming@163.com}
\address{
Lin Zhang\newline\indent Department of Mathematics\\
Zhejiang University\\Hangzhou\\
People's Republic of China}
\email{godyalin@163.com;linyz@zju.edu.cn}

\begin{document}

\title[Direct and Inverse Estimates for Combinations]{Direct and Inverse Estimates for Combinations of Bernstein Polynomials with Endpoint Singularities}
\begin{abstract}
We give direct and inverse theorems  for the weighted approximation
of functions with endpoint singularities by combinations of
Bernstein polynomials by the $r$th Ditzian-Totik modulus of smoothness $\omega_\phi^{r}(f,t)_w$ where $\phi$ is an admissible step-weight function.
\end{abstract}
\maketitle
\section{Introduction} \label{sect1}
The set of all continuous functions, defined on the interval $I$, is
denoted by $C(I)$. For any $f\in C([0,1])$, the corresponding
\emph{Bernstein operators} are defined as follows:
$$B_n(f,x):=\sum_{k=0}^nf(\frac{k}{n})p_{n,k}(x),$$
where
$$p_{n,k}(x):={n \choose k}x^k(1-x)^{n-k}, \ k=0,1,2,\ldots,n, \ x\in[0,1].$$
Approximation properties of Bernstein operators have been studied
very well (see \cite{Berens}, \cite{Della},
\cite{Totik}-\cite{Lorentz}, \cite{Yu}-\cite{X. Zhou}, for example).
In order to approximate the functions with singularities, Della
Vecchia et al. \cite{Della} and Yu-Zhao \cite{Yu} introduced some
kinds of \emph{modified Bernstein operators}. Throughout the paper,
$C$ denotes a positive constant independent of $n$ and $x$, which
may be different in different cases.
Ditzian and Totik
\cite{Totik} extended the method of combinations and defined the
following combinations of Bernstein operators:
\begin{align*}
B_{n,r}(f,x):=\sum_{i=0}^{r-1}C_{i}(n)B_{n_i}(f,x),
\end{align*}
with the conditions:
\begin{enumerate}[(a)]
\item $n=n_0<n_1< \cdots <n_{r-1}\leqslant
Cn,$\\
\item $\sum_{i=0}^{r-1}|C_{i}(n)|\leqslant C,$\\
\item
$\sum_{i=0}^{r-1}C_{i}(n)=1,$\\
\item $\sum_{i=0}^{r-1}C_{i}(n)n_{i}^{-k}=0$,\ for $k=1,\ldots,r-1$.
\end{enumerate}
Now, we can define our new combinations of Bernstein operators as
follows:
\begin{align}
B^\ast_{n,r}(f,x):=B_{n,r}(F_n,x)=\sum_{i=0}^{r-1}C_{i}(n)B_{n_i}(F_{n},x),\label{s1}
\end{align}
where $C_{i}(n)$ satisfy the conditions (a)-(d). For the details, it
can be referred to \cite{S. Yu}.
\section{The main results} \label{ns}
Let $\varphi(x)=\sqrt{x(1-x)}$ and let $\phi:[0,1] \longrightarrow
R,\ \phi \neq 0$ be an admissible step-weight function of the
Ditzian-Totik modulus of smoothness, that is, $\phi$ satisfies the
following conditions:\\
\begin{enumerate}[(I)]
\item For every proper subinterval $[a,b] \subseteq [0,1]$ there
exists a constant $M_1 \equiv M(a,b)>0$ such that $M_{1}^{-1}
\leqslant
\phi(x) \leqslant M_1$ for $x \in [a,b].$\\
\item There are two
numbers $\beta(0) \geqslant 0$ and $\beta(1) \geqslant 0$ for which
\begin{align*}
\phi(x)\thicksim \left\{
\begin{array}{lrr}
x^{\beta(0)},&&\ \mbox{as}\ x \rightarrow 0+,    \\
(1-x)^{\beta(1)},&&\ \mbox{as}\ x\rightarrow 1-.
              \end{array}
\right.\\
\end{align*}
($X \sim Y$ means $C^{-1}Y \leqslant X \leqslant CY \mbox{for\
some}\ C$).
\end{enumerate}
Combining condition (\textbf{I}) and (\textbf{II}) on $\phi,$ we can deduce that\\
\begin{align*}
M^{-1}\phi_2(x) \leqslant \phi(x) \leqslant M\phi_2(x),\ x \in [0,1],
\end{align*}
where $\phi_2(x) = x^{\beta(0)}(1-x)^{\beta(1)},$ and $M$ is a positive constant independent of $x$.\\
Let
\begin{align*}
w(x)=x^\alpha(1-x)^\beta,\ \alpha,\ \beta \geqslant0,\ \alpha +
\beta >0,\ 0 \leqslant x \leqslant 1.
\end{align*}
and
\begin{align*}
C_w:= \{{f \in C((0,1)) :\lim\limits_{x\longrightarrow
1}(wf)(x)=\lim\limits_{x\longrightarrow 0}(wf)(x)=0 }\}.
\end{align*}
The \emph{norm} in $C_w$ is defined by
$\|wf\|_{C_w}:=\|wf\|=\sup\limits_{0\leqslant x\leqslant
1}|(wf)(x)|$. Define
\begin{align*}
W_{\phi}^{r}:= \{f \in C_w:f^{(r-1)} \in A.C.((0,1)),\
\|w\phi^{r}f^{(r)}\|<\infty\},\\
W_{\varphi,\lambda}^{r}:= \{f \in C_w:f^{(r-1)} \in A.C.((0,1)),\
\|w\varphi^{r\lambda}f^{(r)}\|<\infty\}.
\end{align*}
For $f \in C_w$, define the \emph{weighted modulus of smoothness} by
\begin{align*}
\omega_\phi^{r}(f,t)_w:=\sup_{0<h\leqslant
t}\{\|w\triangle_{h\phi}^{r}f\|_{[16h^2,1-16h^2]}+\|w\overrightarrow{\triangle}_{h}^{r}f\|_{[0,16h^2]}+\|w\overleftarrow{\triangle}_{h}^{r}f\|_{[1-16h^2,1]}\},
\end{align*}
where
\begin{align*}
\Delta_{h\phi}^{r}f(x)=\sum_{k=0}^{r}(-1)^{k}{r \choose
k}f(x+(\frac r2-k)h\phi(x)),\\
\overrightarrow{\Delta}_{h}^{r}f(x)=\sum_{k=0}^{r}(-1)^{k}{r
\choose k}f(x+(r-k)h),\\
\overleftarrow{\Delta}_{h}^{r}f(x)=\sum_{k=0}^{r}(-1)^{k}{r \choose
k}f(x-kh).
\end{align*}
Recently Felten showed the following two theorems in
\cite{Felten}:\\~\\
\textbf{Theorem A.} Let $\varphi(x)=\sqrt{x(1-x)}$ and let
$\phi:[0,1] \longrightarrow R,\ \phi \neq 0$ be an admissible
step-weight function of the Ditzian-Totik modulus of
smoothness(\cite{Totik}) such that $\phi^2$ and $\varphi^2/\phi^2$
are concave. Then, for $f \in C[0,1]$ and $0< \alpha <2,\
|B_n(f,x)-f(x)| \leqslant
\omega_\phi^{2}(f,n^{-1/2}{\frac {\varphi(x)}{\phi(x)}}).$\\~\\
\textbf{Theorem B.} Let $\varphi(x)=\sqrt{x(1-x)}$ and let
$\phi:[0,1] \longrightarrow R,\ \phi \neq 0$ be an admissible
step-weight function of the Ditzian-Totik modulus of smoothness such
that $\phi^2$ and $\varphi^2/\phi^2$ are concave. Then, for $f \in
C[0,1]$ and $0< \alpha <2,\ |B_n(f,x)-f(x)|=O((n^{-1/2}{\frac
{\varphi(x)}{\phi(x)}})^\alpha)$ implies
$\omega_\phi^{2}(f,t)=O(t^\alpha).$
\\
Our main results are the following: \\
\begin{thrm}\label{t1}
For any $\alpha, \ \beta >0,\ \min\{\beta(0),\beta(1)\} \geqslant
{\frac 12},\ f \in C_w,$ we have
\begin{align}
|w(x)\phi^{r}(x)B^{\ast(r)}_{n,r-1}(f,x)|\leqslant Cn^{\frac
r2}\|wf\|.\label{s2}
\end{align}
\end{thrm}
\begin{thrm}\label{t2} For any $\alpha, \ \beta >0,\ f\in W_\phi^{r},\ \min\{\beta(0),\beta(1)\} \geqslant
{\frac 12},$ we have
\begin{align}
|w(x)\phi^{r}(x)B^{\ast(r)}_{n,r-1}(f,x)|\leqslant
C\|w\phi^{r}f^{(r)}\|.\label{s3}
\end{align}
\end{thrm}
\begin{thrm}\label{t3} For $f\in C_w,\ \alpha, \ \beta>0,\ \min\{\beta(0),\beta(1)\} \geqslant
{\frac 12},\ \alpha_0 \in (0,r),$ we have
\begin{align}
w(x)|f(x)-B^\ast_{n,r-1}(f,x)|=O((n^{-{\frac
12}}\phi^{-1}(x)\delta_n(x))^{\alpha_0}) \Longleftrightarrow
\omega_{\phi}^r(f,t)_w=O(t^{\alpha_0}).\label{s4}
\end{align}
\end{thrm}
\section{Lemmas}
\begin{lem}(\cite{Zhou}) For any non-negative real $u$ and $v$, we
have
\begin{align}
\sum_{k=1}^{n-1}({\frac kn})^{-u}(1-{\frac
kn})^{-v}p_{n,k}(x)\leqslant Cx^{-u}(1-x)^{-v}.\label{s5}
\end{align}
\end{lem}
\begin{lem} (\cite{Della}) If $\gamma \in R,$ then
\begin{align}
\sum_{k=0}^n|k-nx|^\gamma p_{n,k}(x) \leqslant Cn^{\frac
\gamma2}\varphi^\gamma(x).\label{s6}
\end{align}
\end{lem}
\begin{lem} For any $f\in W_\phi^{r},$  $\alpha,\ \beta
>0$, we have
\begin{align}
\|w\phi^{r}F_{n}^{(r)}\| \leqslant C\|w\phi^{r}f^{(r)}\|.\label{s7}
\end{align}
\end{lem}
\begin{proof} By symmetry, we only prove the above result when
$x\in (0,1/2]$, the others can be done similarly. Obviously, when $x
\in (0,1/n],$ by \cite{Totik}, we have
\begin{eqnarray*}
|L_r^{(r)}(f,x)| &\leqslant& C|\overrightarrow{\Delta}_{\frac
1r}^{r}f(0)| \ \leqslant \ Cn^{-{\frac r2}+1}\int_0^{\frac rn}u^{\frac r2}|f^{(r)}(u)|du\\
&\leqslant& Cn^{-{\frac r2}+1}\|w\phi^{r}f^{(r)}\|\int_0^{\frac
rn}u^{\frac r2}w^{-1}(u)\phi^{-r}(u)du.
\end{eqnarray*}
So
\begin{eqnarray*}
|w(x)\phi^{r}(x)F_{n}^{(r)}(x)| \leqslant C\|w\phi^{r}f^{(r)}\|.
\end{eqnarray*}
If $x\in [{\frac 1n},{\frac 2n}],$ we have
\begin{eqnarray*}
|w(x)\phi^{r}(x)F_{n}^{(r)}(x)| &\leqslant& |w(x)\phi^{r}(x)f^{(r)}(x)| + |w(x)\phi^{r}(x)(f(x)-F_{n}(x))^{(r)}|\\
&:=& I_1 + I_2.
\end{eqnarray*}
For $I_2,$ we have
\begin{eqnarray*}
f(x)-F_{n}(x) &=& (\psi(nx-1)+1)(f(x)-L_r(f,x)).\\
w(x)\phi^{r}(x)|(f(x)-F_{n}(x))^{(r)}|&=&w(x)\phi^{r}(x)\sum_{i=0}^rn^i|(f(x)-L_r(f,x))^{(r-i)}|.
\end{eqnarray*}
By \cite{Totik}, then
\begin{eqnarray*}
|(f(x)-L_r(f,x))^{(r-i)}|_{[{\frac 1n},{\frac 2n}]} \leqslant
C(n^{r-i}\|f-L_r\|_{[{\frac 1n},{\frac 2n}]} +
n^{-i}\|f^{(r)}\|_{[{\frac 1n},{\frac 2n}]}),\ 0<j<r.
\end{eqnarray*}
Now, we estimate
\begin{eqnarray}
I:=w(x)\phi^{r}(x)|f(x)-L_r(x)|.\label{s8}
\end{eqnarray}
By Taylor expansion, we have
\begin{eqnarray}
f({\frac in})=\sum_{u=0}^{r-1}\frac{({\frac
in}-x)^u}{u!}f^{(u)}(x)+{\frac 1{(r-1)!}}\int_{x}^{{\frac
in}}({\frac in}-s)^{r-1}f^{(r)}(s)ds,\label{s9}
\end{eqnarray}
It follows from (\ref{s9}) and the identity
\begin{eqnarray*}
\sum\limits_{i=1}^{r}({\frac in})^{v}l_{i}(x)=Cx^v,\ v=0,1,\cdots,r.
\end{eqnarray*}
we have
\begin{eqnarray*}
L_r(f,x)&=&\sum_{i=1}^{r}\sum_{u=0}^{r-1}\frac{({\frac
in}-x)^u}{u!}f^{(u)}(x)l_{i}(x)+{\frac
1{(r-1)!}}\sum_{i=1}^{r}l_{i}(x)\int_{x}^{{\frac in}}({\frac in}-s)^{r-1}f^{(r)}(s)ds\nonumber\\
&=&f(x)+C\sum_{u=1}^{r-1}f^{(u)}(x)(\sum_{v=0}^{u}C_{u}^{v}(-x)^{u-v}\sum_{i=1}^{r}({\frac in})^{v}l_{i}(x))\nonumber\\
&&+\ \ {\frac 1{(r-1)!}}\sum_{i=1}^{r}l_{i}(x)\int_{x}^{{\frac
in}}({\frac in}-s)^{r-1}f^{(r)}(s)ds,
\end{eqnarray*}
which implies that
$${w(x)}\phi^{r}(x)|f(x)-L_r(f,x)|={\frac 1{(r-1)!}}{w(x)}\phi^{r}(x)\sum_{i=1}^{r}l_{i}(x)\int_{x}^{{\frac in}}({\frac in}-s)^{r-1}f^{(r)}(s)ds,$$
since $|l_{i}(x)|\leqslant C$ for $x\in [0,{\frac 2n}],\
i=1,2,\cdots,r$.
It follows from $\frac{|{\frac in}-s|^{r-1}}{w(s)}\leqslant
\frac{|{\frac in}-x|^{r-1}}{w(x)},$ $s$ between ${\frac in}$ and
$x$, then
\begin{eqnarray*}
w(x)\phi^{r}(x)|f(x)-L_r(f,x)|&\leqslant& Cw(x)\phi^{r}(x)\sum_{i=1}^{r}\int_{x}^{{\frac in}}({\frac in}-s)^{r-1}|f^{(r)}(s)|ds\nonumber\\
&\leqslant& C\phi^{r}(x)\|w\phi^{r}f^{(r)}\|\sum_{i=1}^{r}\int_{x}^{{\frac in}}({\frac in}-s)^{r-1}\phi^{-r}(s)ds\nonumber\\
&\leqslant& {\frac {C}{n^r}}\|w\phi^{r}f^{(r)}\|.
\end{eqnarray*}
Thus $I \leqslant C\|w\phi^{r}f^{(r)}\|$. So, we get $I_2 \leqslant
C\|w\phi^{r}f^{(r)}\|$. Above all, we have
\begin{align*}
|w(x)\phi^{r}(x)F_{n}^{(r)}(x)| \leqslant C\|w\phi^{r}f^{(r)}\|.
\end{align*}
\end{proof}
\begin{lem} If $f\in W_\phi^{r},$  $\alpha,\ \beta
>0,\ \min\{\beta(0),\beta(1)\} \geqslant
{\frac 12},$ then
\begin{eqnarray}
|w(x)(f(x)-L_r(f,x))|_{[0,{\frac 2n}]} \leqslant C(\frac
{\delta_n(x)}{\sqrt{n}\phi(x)})^r\|w\phi^{r}f^{(r)}\|.\label{s10}\\
|w(x)(f(x)-R_r(f,x))|_{[1-{\frac 2n},1]} \leqslant C(\frac
{\delta_n(x)}{\sqrt{n}\phi(x)})^r\|w\phi^{r}f^{(r)}\|.\label{s11}
\end{eqnarray}
\end{lem}
\begin{proof} By Taylor expansion, we have
\begin{eqnarray}
f({\frac in})=\sum_{u=0}^{r-1}\frac{({\frac
in}-x)^u}{u!}f^{(u)}(x)+{\frac 1{(r-1)!}}\int_{x}^{{\frac
in}}({\frac in}-s)^{r-1}f^{(r)}(s)ds,\label{s12}
\end{eqnarray}
It follows from (\ref{s12}) and the identities
\begin{eqnarray*}
\sum\limits_{i=1}^{r-1}({\frac in})^{v}l_{i}(x)=Cx^v,\
v=0,1,\ldots,r.
\end{eqnarray*}
we have
\begin{eqnarray*}
L_r(f,x)&=&\sum_{i=1}^{r}\sum_{u=0}^{r-1}\frac{({\frac
in}-x)^u}{u!}f^{(u)}(x)l_{i}(x)+{\frac
1{(r-1)!}}\sum_{i=1}^{r}l_{i}(x)\int_{x}^{{\frac in}}({\frac in}-s)^{r-1}f^{(r)}(s)ds\nonumber\\
&=&f(x)+C\sum_{u=1}^{r-1}f^{(u)}(x)(\sum_{v=0}^{u}C_{u}^{v}(-x)^{u-v}\sum_{i=1}^{r}({\frac in})^{v}l_{i}(x))\nonumber\\
&&+\ \ {\frac 1{(r-1)!}}\sum_{i=1}^{r}l_{i}(x)\int_{x}^{{\frac
in}}({\frac in}-s)^{r-1}f^{(r)}(s)ds,
\end{eqnarray*}
which implies that
$$w(x)|f(x)-L_r(f,x)|={\frac 1{(r-1)!}}w(x)\sum_{i=1}^{r}l_{i}(x)\int_{x}^{{\frac in}}({\frac in}-s)^{r-1}f^{(r)}(s)ds,$$
since $|l_{i}(x)|\leqslant C$ for $x\in [0,{\frac 2n}],\
i=1,2,\cdots,r$.
It follows from $\frac{|{\frac in}-s|^{r-1}}{w(s)}\leqslant
\frac{|{\frac in}-x|^{r-1}}{w(x)},$ $s$ between ${\frac in}$ and
$x$, then
\begin{eqnarray*}
w(x)|f(x)-L_r(f,x)|&\leqslant& Cw(x)\sum_{i=1}^{r}\int_{x}^{{\frac in}}({\frac in}-s)^{r-1}|f^{(r)}(s)|ds\nonumber\\
&\leqslant& C{\frac
{\varphi^r(x)}{\phi^{r}(x)}}\|w\phi^{r}f^{(r)}\|\sum_{i=1}^{r}\int_{x}^{{\frac in}}({\frac in}-s)^{r-1}\varphi^{-r}(s)ds\nonumber\\
&\leqslant& C{\frac
{\delta_n^r(x)}{\phi^{r}(x)}}\|w\phi^{r}f^{(r)}\|\sum_{i=1}^{r}\int_{x}^{{\frac in}}({\frac in}-s)^{r-1}\varphi^{-r}(s)ds\nonumber\\
&\leqslant& C(\frac
{\delta_n(x)}{\sqrt{n}\phi(x)})^r\|w\phi^{r}f^{(r)}\|.
\end{eqnarray*}
The proof of (\ref{s11}) can be done similarly.
\end{proof}
\begin{lem}(\cite{S. Yu}) For every $\alpha,\ \beta>0,$ we have
\begin{eqnarray}
\|wB^\ast_{n,r-1}(f)\| \leqslant C\|wf\|.\label{s13}
\end{eqnarray}
\end{lem}
\begin{lem}(\cite{Wang}) Let $\min\{\beta(0), \beta(1)\}
\geqslant {\frac 12}$, then for $r \in N,\ 0<t<{\frac {1}{8r}}$ and
${\frac {rt}{2}} < x < 1-{\frac {rt}{2}},$ we have
\begin{eqnarray}
\int_{-{\frac {t}{2}}}^{\frac {t}{2}} \cdots \int_{-{\frac
{t}{2}}}^{\frac {t}{2}}\phi^{-r}(x+\sum_{k=1}^ru_k)du_1 \cdots du_r
\leqslant Ct^r\phi^{-r}(x).\label{s14}
\end{eqnarray}
\end{lem}
\begin{lem}(\cite{Lu}) If \ $\alpha, \ \beta >0,$ for any $f \in C_w,$ we have
\begin{align}
\|wB^{\ast(r)}_{n,r-1}(f)\| \leqslant Cn^{r}\|wf\|.\label{s15}
\end{align}
\end{lem}
\section{Proof of Theorems}
\subsection{Proof of Theorem \ref{t1}}

When $f\in C_w,\ \min\{\beta(0), \beta(1)\}
\geqslant {\frac 12},$ we discuss it as follows:\\~\\
\textit{Case 1.} If $0\leqslant \varphi(x)\leqslant {\frac
{1}{\sqrt{n}}}$, by (\ref{s15}), we have
\begin{align}
|w(x)\phi^{r}(x)B_{n,r-1}^{\ast(r)}(f,x)| = C\varphi^r(x)\cdot{\frac
{\phi^{r}(x)}{\varphi^r(x)}}|w(x)B_{n,r-1}^{\ast(r)}(f,x)|\nonumber\\
\leqslant Cn^{\frac r2}\|wf\|.\label{s16}
\end{align} \textit{Case 2.} If
$\varphi(x)> {\frac {1}{\sqrt{n}}}$, we have
\begin{align*}
|B_{n,r-1}^{\ast(r)}(f,x)|=|B_{n,r-1}^{(r)}(F_{n},x)|\nonumber\\
\leqslant(\varphi^{2}(x))^{-r}\sum_{i=0}^{r-2}\sum_{j=0}^{r}Q_{j}(x,n_i)C_{i}(n)n_{i}^{j}\sum_{k=0}^{n_i}|(x-{\frac
kn_{i}})^{j}F_{n}({\frac kn_{i}})|p_{n_i,k}(x),
\end{align*}
By \cite{Totik}, we have\\ $Q_{j}(x,n_i)=(n_ix(1-x))^{[{\frac
{r-j}{2}}]},$ and
$(\varphi^{2}(x))^{-r}Q_{j}(x,n_i)n_{i}^{j}\leqslant
C(n_i/\varphi^{2}(x))^{\frac {r+j}{2}}$.\\ So
\begin{align}
|w(x)\phi^{r}(x)B_{n,r-1}^{\ast(r)}(f,x)|\nonumber\\
\leqslant Cw(x)\phi^{r}(x)\sum_{i=0}^{r-2}\sum_{j=0}^{r}({\frac
{n_{i}}{\varphi^2(x)}})^{\frac {r+j}{2}}\sum_{k=0}^{n_i}|(x-{\frac
{k}{n_{i}}})^{j}F_{n}({\frac {k}{n_{i}}})|p_{n_i,k}(x)\nonumber\\
\leqslant
Cw(x)\phi^{r}(x)\|wf\|{\sum_{i=0}^{r-2}\sum_{j=0}^{r}({\frac
{n_{i}}{\varphi^2(x)}})^{\frac {r+j}{2}}}{\{\sum_{k=0}^{n_i}(x-{\frac {k}{n_{i}}})^{2j}\}^{\frac 12}} \cdot \nonumber\\
\{\sum_{k=0}^{n_i}w^{-2}({\frac {k}{n_i}})p_{n_i,k}(x)\}^{\frac 12}\nonumber\\
\leqslant Cn^{\frac r2}\|wf\|. \label{s17}
\end{align}
It follows from combining with (\ref{s16}) and (\ref{s17}) that the theorem is proved. $\Box$\\
\subsection{Proof of Theorem \ref{t2}}
When $f\in W_\phi^{r},$ by \cite{Totik}, we have
\begin{align}
B_{n,r-1}^{(r)}(F_{n},x)=\sum_{i=0}^{r-2}C_{i}(n)n_{i}^{r}\sum_{k=0}^{n_{i}-r}\overrightarrow{\Delta}_{\frac
1{n_{i}}}^{r}F_{n}({\frac kn_{i}})p_{n_i-r,k}(x).\label{s18}
\end{align}
If $0<k<n_{i}-r,$ we have
\begin{align}
|\overrightarrow{\Delta}_{\frac 1{n_{i}}}^{r}F_{n}({\frac kn_{i}})
|\leqslant Cn_{i}^{-r+1}\int_{0}^{\frac
{r}{n_{i}}}|F_{n}^{(r)}({\frac kn_{i}}+u)|du,\label{s19}
\end{align}
If $k=0,$ we have
\begin{align}
|\overrightarrow{\Delta}_{\frac 1{n_{i}}}^{r}F_{n}(0)|\leqslant
C\int_{0}^{\frac {r}{n_{i}}}u^{r-1}|F_{n}^{(r)}(u)|du,\label{s20}
\end{align}
Similarly
\begin{align}
|\overrightarrow{\Delta}_{\frac 1{n_{i}}}^{r}F_{n}({\frac
{n_{i}-r}{n_{i}}})|\leqslant Cn_i^{-r+1}\int_{1-{\frac
{r}{n_{i}}}}^{1}(1-u)^{\frac r2}|F_{n}^{(r)}(u)|du.\label{s21}
\end{align}
By (\ref{s18})-(\ref{s21}), we have
\begin{align}
|w(x)\phi^{r}(x)B_{n,r-1}^{\ast(r)}(f,x)|\nonumber\\\leqslant
C{w(x)}\phi^{r}(x)\|w\phi^{r}F_n^{(r)}\|\sum_{i=0}^{r-2}\sum_{k=0}^{n_{i}-r}(w\phi^{r})^{-1}(\frac
{k^\ast}{n_i})p_{n_i-r,k}(x),\label{s22}
\end{align}
where $k^\ast =1$ for $k=0,$ $k^\ast =n_i-r-1$ for $k=n_i-r$ and
$k^\ast =k$ for $1<k<n_i-r.$ By (\ref{s5}), we have
\begin{align*}
\sum_{k=0}^{n_i-r}(w\phi^{r})^{-1}(\frac
{k^\ast}{n_i})p_{n_i-r,k}(x) \leqslant C(w\phi^{r})^{-1}(x).
\end{align*}
which combining with (\ref{s22}) give
\begin{align*}
|w(x)\phi^{r}(x)B_{n,r-1}^{\ast(r)}(f,x)|\leqslant
C\|w\phi^{r}f^{(r)}\|.\Box
\end{align*}
\\
Combining with the theorem \ref{t1} and theorem \ref{t2}, we can
obtain \\~\\
\textbf{Corollary} {\it For any $\alpha, \ \beta
>0,\ 0 \leqslant \lambda \leqslant 1,$ we have
\begin{align}
|w(x)\varphi^{r\lambda}(x)B^{\ast(r)}_{n,r-1}(f,x)|\leqslant \left\{
\begin{array}{lrr}
Cn^{r/2}{\{max\{n^{r(1-\lambda)/2},\varphi^{r(\lambda-1)}(x)\}\}}\|wf\|,\ &&    f\in C_w,    \\
C\|w\varphi^{r\lambda}f^{(r)}\|,\ && f\in
W_{w,\lambda}^{r}.\label{s23}
              \end{array}
\right.
\end{align}}
\subsection{Proof of Theorem \ref{t3}}
\subsubsection{The direct theorem} We know
\begin{align}
F_n(t)=F_n(x)+F'_n(t)(t-x) + \cdots + {\frac{1}{(r-1)!}}\int_x^t
(t-u)^{r-1}F_n^{(r)}(u)du,\label{s24}\\
B_{n,r-1}((\cdot-x)^k,x)=0, \ k=1,2,\cdots,r-1.\label{s25}
\end{align}
According to the definition of $W_\phi^{r},$ \ for any $g \in
W_\phi^{r},$ we have $B^\ast_{n,r-1}(g,x)=B_{n,r-1}(G_{n}(g),x),$
and
$w(x)|G_{n}(x)-B_{n,r-1}(G_{n},x)|=w(x)|B_{n,r-1}(R_r(G_n,t,x),x)|,$
thereof $R_r(G_n,t,x)=\int_x^t (t-u)^{r-1}G^{(r)}_n(u)du,$ we have
\begin{align}
w(x)|G_{n}(x)-B_{n,r-1}(G_{n},x)|  \leqslant
C\|w\phi^{r}G^{(r)}_n\|w(x)B_{n,r-1}({\int_x^t{\frac
{|t-u|^{r-1}}{w(u)\phi^{r}(u)}du,x}})\nonumber\\
 \leqslant
C\|w\phi^{r}G^{(r)}_n\|w(x)(B_{n,r-1}(\int_x^t{\frac
{|t-u|^{r-1}}{\phi^{2r}(u)}}du,x))^{\frac 12}\cdot
\nonumber\\
(B_{n,r-1}(\int_x^t{\frac {|t-u|^{r-1}}{w^2(u)}}du,x))^{\frac
12}.\label{s26}
\end{align}
also
\begin{align}
\int_x^t{\frac {|t-u|^{r-1}}{\phi^{2r}(u)}}du \leqslant C{\frac
{|t-x|^r}{\phi^{2r}(x)}},\ \int_x^t{\frac {|t-u|^{r-1}}{w^2(u)}}du
\leqslant {\frac {|t-x|^r}{w^2(x)}}.\label{s27}
\end{align}
By (\ref{s6}), (\ref{s7}) and (\ref{s27}), we have
\begin{align}
w(x)|G_{n}(x)-B_{n,r-1}(G_{n},x)| \leqslant
C\|w\phi^{r}G^{(r)}_n\|\phi^{-r}(x)B_{n,r-1}
(|t-x|^r,x)\nonumber\\
\leqslant Cn^{-\frac r2}{\frac
{\varphi^{r}(x)}{\phi^{r}(x)}}\|w\phi^{r}G^{(r)}_n\|\nonumber\\
\leqslant Cn^{-\frac r2}{\frac
{\delta_n^r(x)}{\phi^{r}(x)}}\|w\phi^{r}G^{(r)}_n\|\nonumber\\
= C(\frac
{\delta_n(x)}{\sqrt{n}\phi(x)})^r\|w\phi^{r}G^{(r)}_n\|.\label{s28}
\end{align}
By (\ref{s10}), (\ref{s11}) and (\ref{s28}), when $g \in
W_\phi^{r},$ then
\begin{align}
w(x)|g(x)-B^\ast_{n,r-1}(g,x)| \leqslant w(x)|g(x)-G_{n}(g,x)| +
w(x)|G_{n}(g,x)-B^\ast_{n,r-1}(g,x)|\nonumber\\
\leqslant |w(x)(g(x)-L_r(g,x))|_{[0,{\frac 2n}]} +
|w(x)(g(x)-R_r(g,x))|_{[1-{\frac 2n},1]}\nonumber\\  +\ \ C(\frac
{\delta_n(x)}{\sqrt{n}\phi(x)})^r\|w\phi^{r}G^{(r)}_n\|\nonumber\\
\leqslant C(\frac
{\delta_n(x)}{\sqrt{n}\phi(x)})^r\|w\phi^{r}g^{(r)}\|.\label{s29}
\end{align}
For $f \in C_w,$ we choose proper $g \in W_\phi^{r},$ by (\ref{s13})
and (\ref{s29}), then
\begin{align*}
w(x)|f(x)-{B^\ast_{n,r-1}(f,x)}| \leqslant w(x)|f(x)-g(x)| +
w(x)|{B^\ast_{n,r-1}(f-g,x)}|\\ +
w(x)|g(x)-{B^\ast_{n,r-1}(g,x)}|\\
\leqslant C(\|w(f-g)\|+(\frac
{\delta_n(x)}{\sqrt{n}\phi(x)})^r\|w\phi^{r}g^{(r)}\|)\\
\leqslant C\omega_\phi^r(f,\frac {\delta_n(x)}{\sqrt{n}\phi(x)})_w.
\Box
\end{align*}
\subsubsection{The inverse theorem}
We define the weighted main-part modulus for $D=R_+$ by(see
\cite{Totik})
\begin{align*}
\Omega_{\phi}^r(C,f,t)_w = \sup_{0<h \leqslant
t}\|w\Delta_{{h\phi}}^rf\|_{[Ch^\ast,\infty]},\\
\Omega_{\phi}^r(1,f,t)_w = \Omega_{\phi}^r(f,t)_w.
\end{align*}
The main-part $K$-functional is given by
\begin{align*}
K_{r,\phi}(f,t^r)_w=\sup_{0<h \leqslant
t}\inf_g\{\|w(f-g)\|_{[Ch^\ast,\infty]}+t^r\|w\phi^{r}g^{(r)}\|_{[Ch^\ast,\infty]},\
g^{(r-1)} \in A.C.((Ch^\ast,\infty))\}.
\end{align*}
By \cite{Totik}, we have
\begin{align}
C^{-1}\Omega_{\phi}^r(f,t)_w \leqslant \omega_{\phi}^{r}(f,t)_w
\leqslant
C\int_0^t{\frac {\Omega_{\phi}^r(f,\tau)_w}{\tau}}d\tau,\label{s30} \\
C^{-1}K_{r,\phi}(f,t^r)_w \leqslant \Omega_{\phi}^r(f,t)_w \leqslant
CK_{r,\phi}(f,t^r)_w.\label{s31}
\end{align}
\begin{proof} Let $\delta>0,$ by (\ref{s31}), we choose proper $g$ so
that
\begin{align}
\|w(f-g)\| \leqslant C\Omega_{\phi}^r(f,\delta)_w,\
\|w\phi^{r}g^{(r)}\| \leqslant
C\delta^{-r}\Omega_{\phi}^r(f,\delta)_w.\label{s32}
\end{align}
For $r \in N,\ 0<t<{\frac {1}{8r}}$ and ${\frac {rt}{2}} < x <
1-{\frac {rt}{2}},$ we have
\begin{align}
|w(x)\Delta_{h\phi}^rf(x)| \leqslant
|w(x)\Delta_{h\phi}^r(f(x)-{B^\ast_{n,r-1}(f,x)})|+|w(x)\Delta_{h\phi}^rB^\ast_{n,r-1}(f-g,x)|\nonumber\\
 +\ |w(x)\Delta_{h\phi}^r{B^\ast_{n,r-1}(g,x)}|\nonumber\\
\leqslant \sum_{j=0}^rC_r^j(n^{-\frac
12}{\frac {\delta_n(x+({\frac r2}-j)h\phi(x))}{\phi(x+({\frac r2}-j)h\phi(x))}})^{\alpha_0}\nonumber\\
 +\ \ \int_{-{\frac {h\phi(x)}{2}}}^{\frac
{h\phi(x)}{2}}\cdots \int_{-{\frac {h\phi(x)}{2}}}^{\frac
{h\phi(x)}{2}}w(x){B^{\ast(r)}_{n,r-1}(f-g,x+\sum_{k=1}^ru_k)}du_1\cdots
du_r\nonumber\\
 +\ \ \int_{-{\frac {h\phi(x)}{2}}}^{\frac
{h\phi(x)}{2}}\cdots \int_{-{\frac {h\phi(x)}{2}}}^{\frac
{h\phi(x)}{2}}w(x){B^{\ast(r)}_{n,r-1}(g,x+\sum_{k=1}^ru_k)}du_1\cdots
du_r\nonumber\\
:= J_1+J_2+J_3.\label{s33}
\end{align}
Obviously
\begin{align}
J_1 \leqslant C(n^{-\frac
12}\phi^{-1}(x)\delta_n(x))^{\alpha_0}.\label{s34}
\end{align}
By (\ref{s15}) and (\ref{s32}), we have
\begin{align}
J_2 \leqslant Cn^r\|w(f-g)\|\int_{-{\frac {h\phi(x)}{2}}}^{\frac
{h\phi(x)}{2}}\cdots \int_{-{\frac {h\phi(x)}{2}}}^{\frac
{h\phi(x)}{2}}du_1 \cdots du_r\nonumber\\
\leqslant Cn^rh^r\phi^{r}(x)\|w(f-g)\|\nonumber\\
\leqslant Cn^rh^r\phi^{r}(x)\Omega_{\phi}^r(f,\delta)_w.\label{s35}
\end{align}
By the first inequality of (\ref{s23}), we let $\lambda=1,$ and
(\ref{s14}) as well as (\ref{s32}), then
\begin{align}
J_2 \leqslant Cn^{\frac r2}\|w(f-g)\|\int_{-{\frac
{h\phi(x)}{2}}}^{\frac {h\phi(x)}{2}} \cdots \int_{-{\frac
{h\phi(x)}{2}}}^{\frac
{h\phi(x)}{2}}\varphi^{-r}(x+\sum_{k=1}^ru_k)du_1 \cdots du_r\nonumber\\
\leqslant Cn^{\frac r2}h^r\phi^{r}(x)\varphi^{-r}(x)\|w(f-g)\|\nonumber\\
\leqslant Cn^{\frac
r2}h^r\phi^{r}(x)\varphi^{-r}(x)\Omega_{\phi}^r(f,\delta)_w.\label{s36}
\end{align}
By the second inequality of (\ref{s14}) and (\ref{s32}), we have
\begin{align}
J_3 \leqslant C\|w\phi^{r}g^{(r)}\|w(x)\int_{-{\frac
{h\phi(x)}{2}}}^{\frac {h\phi(x)}{2}} \cdots \int_{-{\frac
{h\phi(x)}{2}}}^{\frac
{h\phi(x)}{2}}{w^{-1}(x+\sum_{k=1}^ru_k)}\phi^{-r}(x+\sum_{k=1}^ru_k)du_1 \cdots du_r\nonumber\\
\leqslant Ch^r\|w\phi^{r}g^{(r)}\|\nonumber\\
\leqslant Ch^r\delta^{-r}\Omega_{\phi}^r(f,\delta)_w.\label{s37}
\end{align}
Now, by (\ref{s33})-(\ref{s37}), there exists $M>0$ so that
\begin{align*}
|w(x)\Delta_{h\phi}^rf(x)| \leqslant C((n^{-\frac
12}{\frac {\delta_n(x)}{\phi(x)}})^{\alpha_0}\\
+ \min\{n^{\frac r2}{\frac
{\phi^r(x)}{\varphi^r(x)}},n^r\phi^r(x)\}h^r\Omega_{\phi}^r(f,\delta)_w
+ h^r\delta^{-r}\Omega_{\phi}^r(f,\delta)_w)\\
\leqslant C((n^{-\frac
12}{\frac {\delta_n(x)}{\phi(x)}})^{\alpha_0}\\
+\  h^rM^r(n^{-\frac 12}{\frac {\varphi(x)}{\phi(x)}}\ +\ n^{-\frac
12}{\frac {n^{-1/2}}{\phi(x)}})^{-r}\Omega_{\phi}^r(f,\delta)_w
+ h^r\delta^{-r}\Omega_{\phi}^r(f,\delta)_w)\\
\leqslant C((n^{-\frac
12}{\frac {\delta_n(x)}{\phi(x)}})^{\alpha_0}\\
+ h^rM^r(n^{-\frac 12}{\frac
{\delta_n(x)}{\phi(x)}})^{-r}\Omega_{\phi}^r(f,\delta)_w
+ h^r\delta^{-r}\Omega_{\phi}^r(f,\delta)_w).\\
\end{align*}
When $n \geqslant 2,$ we have
\begin{align*}
n^{-\frac 12}\delta_n(x) < (n-1)^{-\frac 12}\delta_{n-1}(x)
\leqslant \sqrt{2}n^{-\frac 12}\delta_n(x),
\end{align*}
Choosing proper $x,\ \delta,\  n \in N,$ so that
\begin{align*}
n^{-\frac 12}{\frac {\delta_n(x)}{\phi(x)}} \leqslant \delta <
(n-1)^{-\frac 12}{\frac {\delta_{n-1}(x)}{\phi(x)}},
\end{align*}
Therefore
\begin{align*}
|w(x)\Delta_{h\phi}^rf(x)| \leqslant C\{\delta^{\alpha_0} +
h^r\delta^{-r}\Omega_{\phi}^r(f,\delta)_w\}.
\end{align*}
By Borens-Lorentz lemma in \cite{Totik}, we get
\begin{align}
\Omega_{\phi}^r(f,t)_w \leqslant Ct^{\alpha_0}.\label{s38}
\end{align}
So, by (\ref{s30}) and (\ref{s38}), we get
\begin{align*}
\omega_{\phi}^{r}(f,t)_w \leqslant C\int_0^t{\frac
{\Omega_{\phi}^r(f,\tau)_w}{\tau}}d\tau =
C\int_0^t\tau^{\alpha_0-1}d\tau = Ct^{\alpha_0}.
\end{align*}
\end{proof}

\end{document}